\documentclass[11pt,reqno]{amsart}
\usepackage[usenames]{color}
\usepackage{amsmath,verbatim,graphicx,enumerate,fancybox}
\newcommand{\I}{\mathrm{i}}
\newcommand{\D}{\mathrm{d}}

\newcommand{\wh}{\widehat}

\newcommand{\lb}{\left(}

\newcommand{\vp}{\varphi}
\newcommand{\ve}{\varepsilon}

\newcommand{\rb}{\right)}
\newcommand{\PD}{\partial}

\newcommand{\wt}{\widetilde}
\newcommand{\Ac}{\mathcal{A}}

\newcommand{\Lc}{\mathcal{L}}

\newcommand{\Rb}{\mathbb{R}}
\newcommand{\Sb}{\mathbb{S}}

\newcommand{\Beq}{\begin{equation}}
	\newcommand{\Eeq}{\end{equation}}
\newcommand{\beq}{\begin{equation*}}
	\newcommand{\eeq}{\end{equation*}}
\newcommand{\bal}{\begin{align}}
	\newcommand{\eal}{\end{align}}
\renewcommand{\O}{\Omega}

\newcommand{\n}{\nabla}

\newcommand{\A}{\alpha}

\newcommand{\bp}{\begin{prob}}
	\newcommand{\ep}{\end{prob}}
\newcommand{\bpr}{\begin{proof}}
	\newcommand{\epr}{\end{proof}}
\renewcommand{\o}{\omega}

\newcommand{\bel}[1]{\begin{equation}\label{#1}}
	\newcommand{\ee}{\end{equation}}

\newcommand{\rr}{\mathbb{R}}

\newtheorem{theorem}{Theorem}[section]

\newtheorem{lemma}[theorem]{Lemma}
\newtheorem{proposition}[theorem]{Proposition}

\theoremstyle{definition}
\newtheorem{definition}[theorem]{Definition}
\newtheorem{remark}[theorem]{Remark}

\title[Partial data inverse problem]{An inverse problem for the relativistic Schr\"odinger equation with
partial boundary data}
\author[Krishnan and Vashisth]{Venkateswaran P. Krishnan$^{\dagger}$ and Manmohan Vashisth$^{\ddagger}$}
\address{$^{\dagger}$ TIFR Centre for Applicable Mathematics, Bangalore 560065, India. 
\newline\indent E-mail:{\tt \ vkrishnan@tifrbng.res.in}\vspace{2mm}
\newline
	{$^{\ddagger}$ TIFR Centre for Applicable Mathematics, Bangalore 560065, India.
	\newline
	\indent E-mail:{\tt\  manmohanvashisth@gmail.com}}}

\begin{document}
  
  \maketitle
\begin{abstract}
	We study the inverse problem of determining the vector and scalar potentials $\mathcal{A}(t,x)=\left(A_{0},A_{1},\cdots,A_{n}\right)$ and $q(t,x)$, respectively, in the relativistic Schr\"odinger equation
	\begin{equation*}
		\Big{(}\left(\partial_{t}+A_{0}(t,x)\right)^{2}-\sum_{j=1}^{n}\left(\partial_{j}+A_{j}(t,x)\right)^{2}+q(t,x)\Big{)}u(t,x)=0 
	\end{equation*}
	in the region $Q=(0,T)\times\Omega$, where $\Omega$ is a $C^{2}$ bounded domain in $\Rb^{n}$ for $n\geq 3$ and $T>\mbox{diam}(\Omega)$ from  partial data on the boundary $\PD Q$. We prove the unique determination of these potentials modulo a natural gauge invariance for the vector field term. 
\end{abstract}
  

  
 \section{Introduction} \label{Introduction}
 {Let $\Omega$ be a bounded simply connected open subset of $\Rb^{n}$ with $C^{2}$ boundary $\PD \Omega$ where $n\geq 3$.}
 \color{black}For $T>\mbox{diam}(\Omega)$,  let $Q:=(0,T)\times\Omega$ and denote its lateral boundary by $\Sigma:=(0,T)\times\partial\Omega$. Consider the linear hyperbolic partial differential operator of second order with time-dependent coefficients: 
 \begin{align}\label{definition of operator}
 \Lc_{\Ac,q}u:=\Big\{\left(\partial_{t}+A_{0}(t,x)\right)^{2}-\sum_{j=1}^{n}\left(\partial_{j}+A_{j}(t,x)\right)^{2}+q(t,x)\Big\} u=0,\ (t,x)\in Q.
 \end{align}
 We denote $\Ac=(A_{0},\cdots, A_{n})$ and $A=(A_{1},\cdots, A_{n})$. Then $\Ac=(A_{0},A)$. We assume that $\Ac$ is $\Rb^{1+n}$ valued with coefficients in $C_{c}^{\infty}(Q)$ and $q\in L^{\infty}(Q)$.  The operator \eqref{definition of operator} is known as the relativistic Schr\"odinger equation  and appears in quantum mechanics and general relativity \cite[Chap. XII]{Schiff}. In this paper, we study an inverse problem related to this operator. More precisely, we are interested in  determining the coefficients in \eqref{definition of operator} from certain measurements made on suitable subsets of the topological boundary of $Q$.

 Starting with the work of Bukhge\`im and Klibanov \cite{Bukhgeuim_Klibanov_Uniqueness_1981}, there has been extensive work in the literature related to inverse boundary value problems for second order linear hyperbolic PDE. For the case when $\Ac$ is $0$ and $q$ is time-independent, the unique determination of $q$ from full lateral boundary Dirichlet to Neumann data  was addressed by Rakesh and Symes in \cite{Rakesh_Symes_Uniqueness_1988}. 
 Isakov in \cite{Isakov_IP_hyperbolic_Damping_Potential_time-independent_1991} considered the same problem with an additional time-independent time derivative perturbation, that is, with $\Ac=(A_{0}(x),0)$ and $q(x)$ and proved uniqueness results.   
 The results in \cite{Rakesh_Symes_Uniqueness_1988} and \cite{Isakov_IP_hyperbolic_Damping_Potential_time-independent_1991} were proved using geometric optics solutions inspired by the work of Sylvester and Uhlmann \cite{Sylvester_Uhlmann_Calderon_problem_1987}.   For the case of time-independent coefficients, another powerful tool to prove uniqueness results is the boundary control (BC) method pioneered by Belishev, see \cite{Belishev_Recent_progress_BC_Method_2007,Belishev_Reconstruction_Riemannian_Manifold_2010,Belishev_BC_Method_2011}. Later it was developed by Belishev, Kurylev, Katchalov, Lassas, Eskin and others; see \cite{Katchalov_Kurylev_Lassas_Book_2001} and references therein. Eskin in \cite{Eskin_New_Approach_local_2006,Eskin_New_Approcah_global_2007} developed a new approach based on the BC method for determining the time-independent vector and scalar potentials assuming $A_{0}=0$ in \eqref{definition of operator}. Hyperbolic inverse problems for  time-independent coefficients have been extensively studied by Yamamoto and his collaborators as well; see \cite{Anikonov_Cheng_Yamamoto,Bellassoued_Yamamoto,Cipolatti_Yamamoto,Hussein_Lesnic_Yamamoto,Imanuvilov_Yamamoto}.
 
 Inverse problems involving time-dependent first and zeroth order perturbations focusing on the cases $\Ac=0$ or when $\Ac$ is of the form $(A_{0},0)$ have been well studied in prior works. We refer to \cite{Ramm_Sjostrand_IP_wave_equation_potential_1991,
 	Ramm_Rakesh_Property_C_1991,Stefanov1987,Stefanov_Inverse_scattering_potential_time_dependent_1989} for some works in this direction. 
 
 Eskin in \cite{Eskin_IHP_time-dependent_2007} considered full first and zeroth order time-dependent perturbations of the wave equation in a Riemannian manifold set-up and proved uniqueness results (for the first order term, uniqueness modulo a natural gauge invariance) from boundary Dirichlet-to-Neumann data, under the assumption that the coefficients are analytic in time. Salazar removed the analyticity assumption of Eskin in \cite{Salazar_time-dependent_first_order_perturbation_2013}, and proved that the unique determination of vector and scalar potential modulo a natural gauge invariance is possible from Dirichlet-to-Neumann data on the boundary. In a recent work of Stefanov and Yang in \cite{Stefanov-Yang},\footnote{We thank Plamen Stefanov for drawing our attention to the results of this paper.} they proved stability estimates for the recovery of light ray transforms of time-dependent first- and zeroth-order perturbations for the wave equation in a Riemannian manifold setting from certain local Dirichlet to Neumann map. Their results, in particular, would give uniqueness results in suitable subsets of the domain  recovering the vector field term up to a natural gauge invariance and the zeroth-order potential term from this data.

 For the case of time-dependent perturbations, if one is interested in global uniqueness results in a finite time domain,  extra information in addition to Dirichlet-to-Neumann data is required  to prove  uniqueness results. Isakov in \cite{Isakov_Completeness_Product_solutions_IP_1991} proved unique-determination of time-dependent potentials (assuming $\mathcal{A}=0$ in \eqref{definition of operator}) from the data set given by the Dirichlet-to-Neumann data as well as the solution and the time derivative of the solution on the domain at the final time. Recently Kian in \cite{Kian_Damping_partial_data_2016} proved unique determination of time-dependent damping coefficient $A_{0}(t,x)$ (with $\Ac$ of the form, $\Ac=(A_{0},0)$) and the potential $q(t,x)$ from partial Dirichlet to Neumann data together with  information of the solution at the final time.

 In this article, we prove unique determination of time-dependent vector and scalar potentials $\mathcal{A}(t,x)$ and $q(t,x)$  appearing in \eqref{definition of operator} (modulo a gauge invariance for the vector potential) from partial boundary data. Our work is related to the result of Salazar \cite{Salazar_time-dependent_first_order_perturbation_2013} who showed uniqueness results for the relativistic Schr\"odinger operator from full Dirichlet to Neumann boundary data assuming such boundary measurements are available for infinite time. It extends the recent work of Kian \cite{Kian_Damping_partial_data_2016}, since we consider the full time-dependent vector field perturbation, whereas Kian assumes only a time derivative perturbation. We should emphasize that the approach using Carleman estimates combined with geometric optics (GO) solutions for recovering time-dependent perturbations, was first used by Kian in \cite{Kian_Uniqueness_partial_data_2016,Kian_stability_partial_data_2016}. To the best of our knowledge, for time-independent perturbations, combination of these techniques to prove uniqueness results first appeared in \cite{Bellassoued_Yamamoto}  inspired by the work of \cite{Bukhgeim_Uhlmann_Calderon_problem_partial_Cauchy_data_2002} for the elliptic case.
 
 The paper is organized as follows. In \S \ref{Main result statement}, we state the main result of the article. In \S \ref{Carleman Estimate}, we prove the Carleman estimates required to prove the existence of GO solutions, and in \S \ref{GO solutions}, we construct the required GO solutions. In \S \ref{Integral identity}, we derive the integral identity using which, we prove the main theorem in \S \ref{Proof for Uniqueness}.

 \section{Statement of the main result}\label{Main result statement}
 In this section, we state the main result of this article. We begin by stating precisely what we mean by gauge invariance. 
 
 \subsection{Gauge Invariance}\label{Definition of Gauge invariance}
 \begin{definition}
 	The vector potentials $\mathcal{A}^{(1)}, \mathcal{A}^{(2)} \in C_{c}^{\infty}(Q)$ are said to be gauge equivalent if there exists a  $\Phi\in C_{c}^{\infty}(Q)$ such that
 	\[
 	\lb\Ac^{(1)}-\Ac^{(2)}\rb(t,x)=\nabla_{t,x}\Phi(t,x).
 	\]
 \end{definition}
 \begin{proposition}\label{Gauge invariance}
 	Suppose $u_{1}(t,x)$ is a solution to the following IBVP 
 	\begin{align}\label{euqtion of interest with A=A1 and q=q1}
 	\begin{aligned}
 	&\Big[(\partial_{t}+A^{(1)}_{0}(t,x))^{2}-\sum_{j=1}^{n}(\partial_{j}+A^{(1)}_{j}(t,x))^{2}+q_{1}(t,x)\Big]u_{1}(t,x)=0,\ \ (t,x)\in Q\\
 	&\ u_{1}(0,x)=\vp(x), \ \partial_{t}u_{1}(0,x)=\psi(x),\ \ \  x\in\Omega\\
 	& \ u_{1}(t,x)=f(t,x), \ \ (t,x)\in \Sigma
 	\end{aligned}
 	\end{align}
 	and $\Phi(t,x)$ is as defined above, then $u_{2}(t,x)=e^{\Phi(t,x)}u_{1}(t,x)$ satisfies the following IBVP
 	\begin{align}\label{equation of interest with A=A2 and q=q2}
 	\begin{aligned}
 	&\Big[(\partial_{t}+A^{(2)}_{0}(t,x))^{2}-\sum_{j=1}^{n}(\partial_{j}+A^{(2)}_{j}(t,x))^{2}+q_{1}(t,x)\Big]u_{2}(t,x)=0, \ \  (t,x)\in Q\\
 	&\ u_{2}(0,x)=\vp(x),\ \partial_{t}u_{2}(0,x)=\psi(x),\ \  x\in\Omega\\
 	&\  u_{2}(t,x)=f(t,x), \ \ (t,x)\in \Sigma
 	\end{aligned}
 	\end{align}
 	with $\Ac^{(1)}$ and  $\Ac^{(2)}$ gauge equivalent. In addition if $\Lambda_{i}$ for $i=1,2$ are the boundary operators  associated with $u_{i}$  and are defined by 
 	\[
 	\Lambda_{i}(\vp,\psi,f):=\lb u_{i}|_{t=T},\PD_{\nu}u_{i}|_{\Sigma}\rb\ \ \text{for} \ i=1,2\]
 	then
 	\[
 	\Lambda_{1}=\Lambda_{2}.
 	\]
 	\begin{proof}
 		It is straightforward computation to verify that $u_{2}= e^{\Phi}u_{1}$ satisfies \eqref{equation of interest with A=A2 and q=q2}.
 		We also have
 		\begin{align*}
 		\begin{aligned}
 		&u_{2}(T,x)=e^{\Phi(T,x)}u_{1}(T,x)=u_{1}(T,x),\ \ x\in\Omega\\ &\PD_{\nu}u_{2}(t,x)=e^{\Phi(t,x)}\lb \PD_{\nu}u_{1}(t,x)+u_{1}(t,x)\PD_{\nu}\Phi(t,x)\rb= \PD_{\nu}u_{1}(t,x); \  \ (t,x)\in\Sigma
 		\end{aligned}
 		\end{align*}
 		where in the above equation, we have used the fact that $\Phi\in C_{c}^{\infty}(Q)$. We thus have $\Lambda_{1}=\Lambda_{2}$. 
 	\end{proof}
 \end{proposition}
 \color{black}
 \subsection{Statement of the main result}\label{Main result}
 We introduce some notation. Following \cite{Bukhgeim_Uhlmann_Calderon_problem_partial_Cauchy_data_2002},  
 fix an $\omega_{0}\in\mathbb{S}^{n-1}$, and define the $\omega_{0}$-shadowed and $\omega_{0}$-illuminated faces by 
 \begin{align*}
 \partial\Omega_{+,\omega_{0}}:=\left\{x\in\partial\Omega:\ \nu(x)\cdot\omega_{0}\geq 0 \right\},\ \ \partial\Omega_{-,\omega_{0}}:=\left\{x\in\partial\Omega:\ \nu(x)\cdot\omega_{0}\leq 0 \right\}
 \end{align*}
 of $\partial\Omega$, where $\nu(x)$ is outward unit normal to $\partial\Omega$ at $x\in\partial\Omega$. Corresponding to $\partial\Omega_{\pm,\omega_{0}}$, we denote the lateral boundary parts by 
 $\Sigma_{\pm,\omega_{0}}:=(0,T)\times\partial\Omega_{\pm,\omega_{0}}$. We denote by $F=(0,T)\times F'$ and $G=(0,T)\times G'$ where $F'$ and $G'$ are small enough open neighbourhoods of $\partial\Omega_{+,\omega_{0}}$ and  $\partial\Omega_{-,\omega_{0}}$ respectively in $\partial\Omega$. 
 
 Consider the IBVP
 
 \begin{align}\label{Equation for u}
 \begin{cases}
 &\Lc_{\mathcal{A},q}u(t,x)=0;\ (t,x)\in Q\\
 &u(0,x)=\phi(x),\ \partial_{t}u(0,x)=\psi(x); \ x\in\Omega\\
 &u(t,x)=f(t,x),\ (t,x)\in \Sigma.
 \end{cases}
 \end{align} 
 
 For $\phi\in H^{1}(\Omega), \psi\in L^{2}(\Omega)$ and $f\in H^{1}(\Sigma)$, \eqref{Equation for u} has a unique solution $u\in C^{1}([0,T];L^{2}(\Omega))\cap C([0,T];H^{1}(\Omega))$  and furthermore  $\partial_{\nu}u\in L^{2}(\Sigma)$; see \cite{Katchalov_Kurylev_Lassas_Book_2001,Lasiecka_Lions_Triggiani_Nonhomogeneous_BVP_hyperbolic_1986}. Thus we have $u\in H^{1}(Q)$. Therefore we can define our input-output operator 
 $\Lambda_{\Ac,q}$ by 
 \begin{align}\label{Definition of input-output operator}
 \Lambda_{\Ac,q}(\phi,\psi,f)=(\partial_{\nu}u|_{G},u|_{t=T})
 \end{align} 
 where $u$ is the solution to \eqref{Equation for u}. The operator 
 \[\Lambda_{\Ac,q}: H^{1}(\Omega)\times L^{2}(\Omega)\times H^{1}(\Sigma)\ \rightarrow \ L^{2}(G) \times H^{1}(\Omega)\] 
 \color{black} is a continuous linear map which follows from the well-posedness of the IBVP given by Equation \eqref{Equation for u} (see \cite{Katchalov_Kurylev_Lassas_Book_2001,Lasiecka_Lions_Triggiani_Nonhomogeneous_BVP_hyperbolic_1986}). A natural question is whether this input-output operator uniquely determines the time-dependent perturbations $\Ac$ and $q$. We now state our main result.
 \begin{theorem}\label{Main Theorem}
 	Let $\left(\mathcal{A}^{(1)},q_{1}\right)$ and $\left(\mathcal{A}^{(2)},q_{2}\right)$ be two sets of vector and scalar potentials 
 	such that each $A_{j}^{(i)}\in C^{\infty}_{c}(Q)$ and $q_{i}\in L^{\infty}(Q)$ for $i=1,2$ and $0\leq j\leq n$. Let $u_{i}$ be solutions to \eqref{Equation for u} when $(\Ac,q)=(\Ac^{(i)},q_{i})$ and  $\Lambda_{\Ac^{(i)},q_{i}}$ for $i=1,2$ be the input-output operators  defined by \eqref{Definition of input-output operator} corresponding to $u_{i}$. If
 	\[\Lambda_{\Ac^{(1)},q_{1}}(\phi,\psi,f)=\Lambda_{\Ac^{(2)},q_{2}}(\phi,\psi,f),\ \text{for all} \ (\phi,\psi,f)\in H^{1}(\Omega)\times L^{2}(\Omega)\times H^{1}(\Sigma),\]
 	then
 	there exists a function $\Phi\in C_{c}^{\infty}(Q)$ such that 
 	\begin{align*}
 	(\mathcal{A}^{(1)}-\mathcal{A}^{(2)})(t,x)=\nabla_{t,x}\Phi(t,x) \text{ and }\ q_{1}(t,x)=q_{2}(t,x),\ (t,x)\in Q.
 	\end{align*}
 	
 \end{theorem}

 \begin{remark}
 	We have stated the above result for vector potentials in $C_{c}^{\infty}(Q)$ for simplicity. It is straightforward to adapt these techniques to prove for the case $\Ac\in W^{2,\infty}(Q)$ that are identical on the boundary.  Using suitable modification of techniques from \cite{Kian_Damping_partial_data_2016}, we believe it can also be proved for $\Ac\in W^{1,\infty}(Q)$ provided they are identical on the boundary, although we have not pursued it in this work.
 \end{remark}
 
 \color{black}
 \section{Carleman Estimate}\label{Carleman Estimate}

 We denote by $H_{\mathrm{scl}}^{1}(Q)$, the semiclassical Sobolev space of order $1$ on $Q$ with the following norm
 \[
 \left|\left|u\right|\right|_{H^{1}_{\mathrm{scl}}(Q)}=\left|\left|u\right|\right|_{L^{2}(Q)}+\left|\left|h\nabla_{t,x}u\right|\right|_{L^{2}(Q)},
 \]
 and for $Q=\mathbb{R}^{1+n}$ we denote by $H^{s}_{\mathrm{scl}}(\mathbb{R}^{1+n})$, the Sobolev space of order $s$ with the norm given by 
 \[
 \left|\left|u\right|\right|_{H^{s}_{\mathrm{scl}}(\rr^{1+n})}^{2}=\left|\left|\left<hD\right>^{s}u\right|\right|_{L^{2}(\mathbb{R}^{1+n})}^{2}=\int\limits_{\mathbb{R}^{1+n}}(1+h^{2}\tau^{2}+h^{2}|\xi|^{2})^{s}\left|\widehat{u}(\tau,\xi)\right|^{2}\D \tau \D \xi.
 \]
 In this section, we derive a Carleman estimate involving boundary terms for \eqref{definition of operator} conjugated with a linear weight. We use this estimate to control boundary terms over subsets of the boundary where measurements are not available. Our proof follows from modifications of the Carleman estimate given in \cite{Kian_Damping_partial_data_2016}. Since we work in a semiclassical setting, we prefer to give the proof for the sake of completeness.
 
 \begin{theorem}\label{Boundary Carleman estimate}
 	Let $\vp(t,x):=t+x\cdot\omega$, where $\omega\in \Sb^{n-1}$ is fixed. Assume that $A_{j}\in C^{\infty}_{c}(Q)$ for $0\leq j\leq n$ and $q\in L^{\infty}(Q)$. Then the  Carleman estimate 
 	\begin{align}\label{Boundary carleman estimate}
 	\notag &h\left(  e^{-\vp/h}\partial_{\nu}\vp\partial_{\nu}u,e^{-\vp/h}\partial_{\nu}u\right)_{L^{2}(\Sigma_{+,\omega})}+h\left( e^{-\frac{\vp(T,\cdot)}{h}}\partial_{t}u(T,\cdot),e^{-\frac{\vp(T,\cdot)}{h}}\partial_{t}u(T,\cdot)\right)_{L^{2}(\Omega)}\\
 	\notag & + \lVert e^{-\vp/h}u\rVert^{2}_{L^{2}(Q)}+\lVert he^{-\vp/h}\partial_{t}u\rVert^{2}_{L^{2}(Q)}+\lVert he^{-\vp/h}\nabla_{x}u\rVert^{2}_{L^{2}(Q)}\\
 	&\leq C\Big( \lVert h e^{-\vp/h}\mathcal{L}_{\mathcal{A},q}u\rVert ^{2}_{L^{2}(Q)}+\left( e^{-\frac{\vp(T,\cdot)}{h}}u(T,\cdot),e^{-\frac{\vp(T,\cdot)}{h}}u(T,\cdot)\right)_{L^{2}(\Omega)}
 	\\
 	\notag & +h\left( e^{-\frac{\vp(T,\cdot)}{h}}\nabla_{x}u(T,\cdot),e^{-\frac{\vp(T,\cdot)}{h}}\nabla_{x}u(T,\cdot)\right)_{L^{2}(\Omega)}+h\left(  e^{-\vp/h}\lb-\partial_{\nu}\vp\rb\partial_{\nu}u,e^{-\vp/h}\partial_{\nu}u\right)_{L^{2}\lb\Sigma_{-,\omega}\rb}\Big)
 	\end{align}
 	holds for all $u\in C^{2}(Q)$ with
 	\begin{align*}
 	u|_{\Sigma}=0,\ u|_{t=0}=\partial_{t}u|_{t=0}=0,
 	\end{align*} 
 	and $h$ small enough.
 \end{theorem}
 \begin{proof}
 	To prove the estimate \eqref{Boundary carleman estimate}, we will use a convexification argument used in \cite{Kian_Damping_partial_data_2016}. Consider the following perturbed weight function
 	\begin{align}\label{value of convexified Carleman weight}
 	\wt{\vp}(t,x)=\vp(t,x)-\frac{ht^{2}}{2\ve}.
 	\end{align}
 	We first consider the conjugated operator
 	\begin{equation}\label{Conjugated Box}
 	\Box_{\vp,\ve}:=h^{2}e^{-\wt{\vp}/h}\Box e^{\wt{\vp}/h}.
 	\end{equation}
 	For $v\in C^{2}(Q)$ satisfying $v|_{\Sigma}=v|_{t=0}=\PD_{t}v|_{t=0}=0$, consider the $L^{2}$ norm of $\Box_{\vp,\ve}$:
 	\begin{align*}\label{L2 norm of conjugated operator}
 	\int\limits_{Q}\left| \Box_{\vp,\ve}v(t,x)\right|^{2}\D x \D t.
 	\end{align*}
 	Expanding \eqref{Conjugated Box}, we get, 
 	\[
 	\Box_{\vp,\ve}v(t,x)=\Big{(} h^{2}\Box  +h \Box \wt{\vp}+\lb \left|\PD_{t} \wt{\vp}\right|^{2}-\left|\n_{x}\wt{\vp}\right|^{2}\rb+2h\lb \PD_{t}\wt{\vp} \PD_{t}-\n_{x}\wt{\vp}\cdot \n_{x}\rb\Big{)} v(t,x).
 	\]
 	We write this as 
 	\[
 	\Box_{\vp,\ve}v(t,x)=P_{1}v(t,x)+P_{2}v(t,x),
 	\]
 	where 
 	\begin{align*}
 	P_{1}v(t,x)& =\Big{(}h^{2}\Box  +h \Box \wt{\vp}+\lb \left|\PD_{t} \wt{\vp}\right|^{2}-\left|\n_{x}\wt{\vp}\right|^{2}\rb\Big{)} v(t,x)\\
 	&=\lb h^{2}\Box + \frac{h^2t^2}{\ve^2}-\frac{2ht}{\ve}-\frac{h^2}{\ve}\rb v(t,x),
 	\end{align*}
 	and 
 	\begin{align*}
 	P_{2}v(t,x)&=2h\lb \PD_{t}\wt{\vp} \PD_{t}-\n_{x}\wt{\vp}\cdot \n_{x}\rb v(t,x)\\
 	&=2h \lb \lb 1-\frac{ht}{\ve}\rb \PD_{t}-\omega\cdot \n_{x}\rb v(t,x).
 	\end{align*}
 	Now 
 	\begin{align*}
 	\int\limits_{Q}&\left|\Box_{\vp,\ve}v(t,x)\right|^{2}\D x \D t\geq 2\int\limits_{Q} \mbox{Re}\lb P_{1}v(t,x)\overline{P_{2}v(t,x)} \rb \D x \D t\\
 	&=4h^{3}\int\limits_{Q}\mbox{Re}\lb \Box v(t,x)\left(1-\frac{ht}{\ve}\right)\overline{\partial_{t}v(t,x)}\rb \D x \D t
 	-4h^{3}\int\limits_{Q}\mbox{Re}\lb \Box v(t,x)\omega\cdot\overline{\nabla_{x}v(t,x)}\rb \D x \D t\\
 	&\quad\quad +4h\int\limits_{Q}\mbox{Re}\lb \left(\frac{h^{2}t^{2}}{\ve^{2}}
 	-\frac{2ht}{\ve}-\frac{h^{2}}{\ve}\right)v(t,x)\left(1-\frac{ht}{\ve}\right)\overline{\partial_{t}v(t,x)}\rb \D x\D t\\
 	&\quad\quad -4h\int\limits_{Q}\mbox{Re}\lb \left(\frac{h^{2}t^{2}}{\ve^{2}}-\frac{2ht}{\ve}-\frac{h^{2}}{\ve}\right)v(t,x)\omega\cdot\overline{\nabla_{x}v(t,x)}\rb \D x\D t\\
 	&:=I_{1}+I_{2}+I_{3}+I_{4}.
 	\end{align*}
 	We first simplify $I_{1}$. We have 
 	\begin{align*}
 	I_{1}=2h^{3}\lb 1-\frac{hT}{\ve}\rb \int\limits_{\Omega} \lb |\PD_{t}v(T,x)|^{2} +|\n_{x}v(T,x)|^{2} \rb \D x +\frac{2h^{4}}{\ve}\int\limits_{Q} \lb |\PD_{t}v(t,x)|^{2}+|\n_{x}v(t,x)|^{2}\rb \D x \D t.
 	\end{align*}
 	In the above derivation, we used integration by parts combined with the hypotheses that $v|_{\Sigma}=v|_{t=0}=\PD_{t}v|_{t=0}=0$. Note that $v|_{\Sigma}=0$ would imply that $\PD_{t}v=0$ on $\Sigma$.
 	
 	Next we consider $I_{2}$. We have 
 	\begin{align*}
 	I_{2}&=-4h^{3} \int\limits_{Q}\mbox{Re}\lb \Box  v(t,x) \omega\cdot \overline{\n_{x}v(t,x)}\rb \D x \D t\\
 	&=-4h^{3}\mbox{Re}\int\limits_{\Omega} \PD_{t} v(T,x) \overline{\omega\cdot \n_{x}v(T,x)} \D x +2h^{3} \int\limits_{\Sigma}\omega\cdot \nu |\PD_{t}v(t,x)|^{2} \D S_{x} \D t +2h^{3}\int\limits_{\Sigma} \omega\cdot \nu |\PD_{\nu}v|^{2} \D S_{x} \D t.
 	\end{align*}
 	In deriving the above equation, we used the fact that 
 	\[
 	2h^{3}\mbox{Re} \int\limits_{\Sigma} \PD_{\nu} v(t,x) \overline{\omega\cdot \n_{x}v(t,x)} \D S_{x} \D t= 2h^{3}\int\limits_{\Sigma} \omega\cdot \nu |\PD_{\nu} v|^{2} \D S_{x} \D t.
 	\]

 	Next we consider $I_{3}$. We have 
 	\begin{align*}
 	I_{3}& =4h\int\limits_{Q}\mbox{Re}\lb \left(\frac{h^{2}t^{2}}{\ve^{2}}
 	-\frac{2ht}{\ve}-\frac{h^{2}}{\ve}\right)v(t,x)\left(1-\frac{ht}{\ve}\right)\overline{\partial_{t}v(t,x)}\rb \D x\D t\\
 	&= 2\int\limits_{\Omega} \left(\frac{h^{3}T^{2}}{\ve^{2}}
 	-\frac{2h^{2}T}{\ve}-\frac{h^{3}}{\ve}\right)\lb 1-\frac{hT}{\ve}\rb |v(T,x)|^{2} \D x\\
 	&\quad \quad -2\int\limits_{Q}\left[ \left( \frac{2h^{3}t}{\ve^{2}}-\frac{2h^{2}}{\ve}\right)\left ( 1-\frac{ht}{\ve}\right)-\frac{h^{2}}{\ve}\left( \frac{h^{2}t^{2}}{\ve^{2}}-\frac{2ht}{\ve}-\frac{h^{2}}{\ve}\right)\right]|v(t,x)|^{2}\D x \D t.
 	\end{align*}
 	Finally, since $v=0$ on $\Sigma$, $I_{4}=0$.
 	
 	Therefore,
 	choosing $\ve$ and $h$ small enough, we have
 	\begin{align}\label{Estimate for P1P2 boundary}
 	\notag \int\limits_{Q} |\Box_{\vp,\ve}v(t,x)|^{2}\D x \D t &\geq \frac{2h^{4}}{\ve}\lb\int\limits_{Q} |\partial_{t}v(t,x)|^{2}
 	+|\nabla_{x}v(t,x)|^{2}\D x \D t\rb+ch^{3}\int\limits_{\Omega}|\partial_{t}v(T,x)|^{2}\D x\\
 	\notag &+2h^{3}\int\limits_{\Sigma}\omega\cdot\nu(x)|\partial_{\nu}v(t,x)|^{2}\D S_{x}\D t-ch^{3}\int\limits_{\Omega}|\nabla_{x}v(T,x)|^{2}\D x\\
 	&-ch^{2}\int\limits_{\Omega}|v(T,x)|^{2}\D x+\frac{ch^{2}}{\ve}\int\limits_{Q}|v(t,x)|^{2}\D x \D t.
 	\end{align}

 	Now we consider the conjugated operator $\Lc_{\vp, \ve}:=h^{2} e^{-\frac{\wt{\vp}}{h}}\Lc_{\Ac,q} e^{\frac{\wt{\vp}}{h}}$. We have  
 	\[
 	\Lc_{\vp, \ve}v(t,x)=h^{2}\lb e^{-\wt{\vp}/h}\left(\Box+2A_{0}\partial_{t}-2A\cdot\nabla_{x}+\wt{q}\right) e^{\wt{\vp}/h}v(t,x)\rb,
 	\]
 	where 
 	\[
 	\wt{q}= q+|A_{0}|^{2}-|A|^{2}+\PD_{t}A_{0}-\n_{x}\cdot A.
 	\]
 	We write
 	\[
 	\Lc_{\vp, \ve}v(t,x)= \Box_{\vp,\ve}v(t,x)+\wt{P} v(t,x),
 	\]
 	where 
 	\begin{equation}\label{Remainder expression}
 	\wt{P}v(t,x)=h^{2}\lb e^{-\wt{\vp}/h}\left(2A_{0}\partial_{t}-2A\cdot\nabla_{x}+\wt{q}\right) e^{\wt{\vp}/h}v(t,x)\rb.
 	\end{equation}
 	By triangle inequality,
 	\begin{equation}\label{Full conjugated}
 	\int\limits_{Q}\left|\Lc_{\vp, \ve}v(t,x)\right|^{2}\D x \D t\geq \frac{1}{2}\int\limits_{Q}|\Box_{\vp,\ve}v(t,x)|^{2}\D x \D t-\int\limits_{Q}|\wt{P}v(t,x)|^{2}\D x \D t.
 	\end{equation}
 	Choosing $h$ small enough, we have,
 	\begin{align}\label{Perturbation terms estimate}
 	\begin{aligned}
 	\int\limits_{Q}\left|\wt{P}v(t,x)\right|^{2}\D x \D t& \leq 
 	Ch^{4}\lb \lVert A_{0}\rVert^{2}_{L^{\infty}(Q)}\int\limits_{Q}|\partial_{t}v(t,x)|^{2}\D x \D t
 	+\lVert A\rVert^{2}_{L^{\infty}(Q)}\int\limits_{Q}|\nabla_{x}v(t,x)|^{2}\D x \D t \rb \\
 	&\quad+Ch^{2} \lb \lVert A_{0}\rVert^{2}_{L^{\infty}(Q)}+\lVert A\rVert^{2}_{L^{\infty}(Q)} \rb \int\limits_{Q}|v(t,x)|^{2}\D x \D t\\
 	&\quad+ Ch^{4}\lVert \wt{q}\rVert^{2}_{L^{\infty}(Q)}\int\limits_{Q}|v(t,x)|^{2}\D x \D t.
 	\end{aligned}
 	\end{align}
 	
 	Using \eqref{Estimate for P1P2 boundary} and \eqref{Perturbation terms estimate} in \eqref{Full conjugated} and taking $\ve$ small enough, we have that there exists a $C>0$ depending only on $\ve$, $T$, $\Omega$, $\Ac$ and $q$  such that
 	\begin{align}\label{Estimate for conjugated operator with convexified weight}
 	\notag &C\left(\int\limits_{Q}\left( |h\partial_{t}v(t,x)|^{2}
 	+|h\nabla_{x}v(t,x)|^{2}\right)\D x \D t+\int\limits_{Q}|v(t,x)|^{2}\D x \D t\right)\\
 	\notag &+Ch\int\limits_{\Omega}|\partial_{t}v(T,x)|^{2}\D x+h\int\limits_{\Sigma}\omega\cdot\nu(x)|\partial_{\nu}v(t,x)|^{2}\D S_{x}\D t\\
 	&\leq \frac{1}{h^{2}}\int\limits_{Q}\left|\Lc_{\vp,\ve}v(t,x)\right|^{2}\D x\D t
 	+Ch\int\limits_{\Omega}|\nabla_{x}v(T,x)|^{2}\D x
 	+C\int\limits_{\Omega}|v(T,x)|^{2}\D x.
 	\end{align}
 	
 	After substituting $v(t,x)=e^{-\frac{\wt{\vp}}{h}} u(t,x)$, we get
 	\begin{align*}
 	&h\left(  e^{-\vp/h}\partial_{\nu}\vp\partial_{\nu}u,e^{-\vp/h}\partial_{\nu}u\right)_{L^{2}(\Sigma_{+,\omega})}+h\left( e^{-\frac{\vp(T,\cdot)}{h}}\partial_{t}u(T,\cdot),e^{-\frac{\vp(T,\cdot)}{h}}\partial_{t}u(T,\cdot)\right)_{L^{2}(\Omega)}\\
 	& + \lVert e^{-\vp/h}u\rVert^{2}_{L^{2}(Q)}+\lVert he^{-\vp/h}\partial_{t}u\rVert^{2}_{L^{2}(Q)}+\lVert he^{-\vp/h}\nabla_{x}u\rVert^{2}_{L^{2}(Q)}\\
 	&\leq C\Big( \lVert h e^{-\vp/h}\mathcal{L}_{\mathcal{A},q}u\rVert ^{2}_{L^{2}(Q)}+\left( e^{-\frac{\vp(T,\cdot)}{h}}u(T,\cdot),e^{-\frac{\vp(T,\cdot)}{h}}u(T,\cdot)\right)_{L^{2}(\Omega)}
 	\\
 	& +h\left( e^{-\frac{\vp(T,\cdot)}{h}}\nabla_{x}u(T,\cdot),e^{-\frac{\vp(T,\cdot)}{h}}\nabla_{x}u(T,\cdot)\right)_{L^{2}(\Omega)}+h\left(  e^{-\vp/h}\lb-\partial_{\nu}\vp\rb\partial_{\nu}u,e^{-\vp/h}\partial_{\nu}u\right)_{L^{2}\lb\Sigma_{-,\omega}\rb}\Big)
 	\end{align*}
 	This completes the proof.
 \end{proof}
 In particular, it follows from the previous calculations that  for $u\in C_{c}^{\infty}(Q)$,
 \begin{equation}\label{Interior Carleman estimate}
 \lVert u\rVert_{H^{1}_{\mathrm{scl}}(Q)}\leq \frac{C}{h}\lVert \Lc_{\vp}u\rVert_{L^{2}(Q)},
 \end{equation}
 where 
 	\[
 	\Lc_{\vp}:=h^{2} e^{-\frac{\vp}{h}}\Lc_{\Ac,q} e^{\frac{\vp}{h}}.
 	\]
 
 {The formal adjoint $\Lc_{\vp}^{*}$ is of the same form as $\Lc_{\vp}$. More precisely,
 	\[
 	\Lc_{\vp}^{*}:=h^{2} e^{\frac{\vp}{h}}\Lc_{-\Ac,\overline{q}} e^{-\frac{\vp}{h}},
 	\]
 	since we have assumed $\Ac$ is real valued. Also note that \eqref{Interior Carleman estimate} would also hold for the operator $\Lc_{\vp}^{*}$.
 }
 
 To show the existence of suitable solutions to \eqref{definition of operator}, we need to shift the Sobolev index by $-1$ in \eqref{Interior Carleman estimate}. This we do in the next lemma. 
 \begin{lemma}\label{weighted H(-1) Carleman estimate}
 	Let $\vp(t,x)=t+x\cdot \omega$ and $\Lc_{\vp}:= h^{2}e^{-\vp/h}\Lc_{\mathcal{A},q}e^{\vp/h}$. There exists an  $h_{0}>0$ such that 
 	\begin{equation}\label{Carleman estimate in Sobolev space of negative order}
 	\lVert v\rVert_{L^{2}(\mathbb{R}^{1+n})}\leq \frac{C}{h} \lVert \Lc_{\vp}v\rVert_{H_{\mathrm{scl}}^{-1}(\mathbb{R}^{1+n})},
 	\end{equation}
 	and 
 	\begin{equation}\label{Carleman estimate in Sobolev space of negative order adjoint operator}
 	\lVert v\rVert_{L^{2}(\mathbb{R}^{1+n})}\leq \frac{C}{h} \lVert \Lc^{*}_{\vp}v\rVert_{H_{\mathrm{scl}}^{-1}(\mathbb{R}^{1+n})} 
 	\end{equation}
 	for all $v\in C_{c}^{\infty}(Q),\ 0< h\leq h_{0}$.
 \end{lemma}
 \begin{proof}We give the proof of the estimate in  \eqref{Carleman estimate in Sobolev space of negative order} and that of \eqref{Carleman estimate in Sobolev space of negative order adjoint operator} follows similarly.
 	We follow arguments used in \cite{Dos-Santos_Kenig_Sajostrand_Uhlmann_Magnetic_schrodinger_partial_data_2007}. We again consider the convexified weight 
 	\[
 	\wt{\vp}(t,x)=t+x\cdot \omega-\frac{ht^{2}}{2\ve},
 	\]
 	and as before consider the convexified operator: 
 	\begin{align*}
 	\Box_{\vp,\ve}:=h^{2}e^{-\wt{\vp}/h}\Box e^{\wt{\vp}/h}.
 	\end{align*}
 	From the properties of pseudodifferential operators, we have 
 	\begin{eqnarray*}
 		\left<hD\right>^{-1}(\Box_{\vp,\ve})\left<hD\right>=\Box_{\vp,\ve}+hR_{1}
 	\end{eqnarray*}
 	where $R_{1}$ is a semi-classical pseudo-differential operator of order $1$. Now 
 	\[
 	\lVert \Box_{\vp,\ve} \langle hD\rangle v\rVert_{H^{-1}_{\mathrm{scl}}(\Rb^{1+n})}=\lVert\langle hD\rangle^{-1} \Box_{\vp,\ve} \langle hD\rangle v\lVert_{L^{2}(\Rb^{1+n})}.
 	\]
 	and by the commutator property above, we get 
 	\[
 	\lVert \Box_{\vp,\ve} \langle hD\rangle v\rVert_{H^{-1}_{\mathrm{scl}}(\Rb^{1+n})}^{2}=\lVert \lb \Box_{\vp,\ve} +hR_{1}\rb v\rVert_{L^{2}(\Rb^{1+n})}^{2}\geq \frac{1}{2}\lVert \Box_{\vp,\ve} v\rVert_{L^{2}(\Rb^{1+n})}^{2}-\lVert hR_{1}v\rVert_{L^{2}(\Rb^{1+n})}^{2}.
 	\]
 	Let $Q\subset\subset \wt{Q}$, and for $v\in C_{c}^{\infty}(\wt{Q})$, using the estimate in \eqref{Estimate for P1P2 boundary} for $C_{c}^{\infty}$ functions combined with estimates for pseudodifferential operators, we have,
 	\begin{equation}\label{H-1 Carleman estimate for Wave}
 	\lVert \Box_{\vp,\ve} \langle hD\rangle v\rVert_{H^{-1}_{\mathrm{scl}}(\Rb^{1+n})}^{2}\geq \frac{Ch^{2}}{\ve} \lVert v\rVert_{H^{1}_{\mathrm{scl}}(\Rb^{1+n})}^{2}-h^{2}\lVert v\rVert_{H^{1}_{\mathrm{scl}}(\Rb^{1+n})}^{2}.
 	\end{equation}		 
 	
 	Using the expression for $\wt{P}$ (see \eqref{Remainder expression}), we get, for $v\in C_{c}^{\infty}(\wt{Q})$ and for $h$ small enough,
 	\[
 	\lVert \wt{P}v\rVert_{H^{-1}_{\mathrm{scl}}(\Rb^{1+n})}^{2}\leq C h^{2}\lVert v\rVert_{L^{2}(\Rb^{1+n})}^{2},
 	\]
 	and therefore
 	\[
 	\lVert \wt{P}\langle hD\rangle v\rVert_{H^{-1}_{\mathrm{scl}}(\Rb^{1+n})}^{2}\leq C h^{2}\lVert \langle hD\rangle v\rVert_{L^{2}(\Rb^{1+n})}^{2}=C h^{2} \lVert v\rVert_{H^{1}_{\mathrm{scl}}(\Rb^{1+n})}^{2}.
 	\]
 	Combining this with the estimate in \eqref{H-1 Carleman estimate for Wave} together with triangle inequality, we get, 
 	\begin{equation}\label{H-1 estimate 2}
 	\lVert \Lc_{\vp,\ve} \langle hD\rangle v\rVert_{H^{-1}_{\mathrm{scl}}(\Rb^{1+n})}^{2}\geq \frac{Ch^{2}}{\ve}\lVert v\rVert_{H^{1}_{\mathrm{scl}}(\Rb^{1+n})}^{2}
 	\end{equation}
 	for all $v\in C_{c}^{\infty}(\wt{Q})$. 
 	
 	Now to complete the proof, for any $u\in C_{c}^{\infty}(Q)$, consider $v=\chi \langle hD\rangle^{-1}u$, where $\chi\in C_{c}^{\infty}(\wt{Q})$ with $\chi \equiv 1$ on $Q$. Then from \eqref{H-1 estimate 2}, we have
 	\[
 	\frac{Ch^{2}}{\ve}\lVert \chi \langle hD\rangle^{-1}u \rVert_{H^{1}_{\mathrm{scl}}(\Rb^{1+n})}^{2}\leq \lVert \Lc_{\vp,\ve} \langle hD\rangle \chi \langle hD\rangle^{-1}u\rVert_{H^{-1}_{\mathrm{scl}}(\Rb^{1+n})}^{2}.
 	\]
 	The operator $\langle hD\rangle \chi \langle hD\rangle^{-1}$ is a semiclassical pseudodifferential operator of order $0$, and therefore we have
 	\[
 	\Lc_{\vp,\ve} \langle hD\rangle \chi \langle hD\rangle^{-1}u= \langle hD\rangle \chi \langle hD\rangle^{-1}\Lc_{\vp,\ve} +hR_{1},
 	\]
 	where $R_{1}$ is a semiclassical pseudodifferential operator of order $1$. 
 	\[
 	\frac{Ch^{2}}{\ve}\lVert \chi \langle hD\rangle^{-1}u \rVert_{H^{1}_{\mathrm{scl}}(\Rb^{1+n})}^{2}\leq \lVert \Lc_{\vp,\ve}u\rVert_{H^{-1}_{\mathrm{scl}}(\Rb^{1+n})}^{2} +h^{2}\lVert u\rVert_{L^{2}(\Rb^{1+n})}^{2}.
 	\]
 	Finally, write 
 	\[
 	\langle hD\rangle ^{-1}u=\chi \langle hD\rangle^{-1}u+(1-\chi) \langle hD\rangle^{-1}u,
 	\]
 	where $\chi$ is as above.
 	Then
 	\[
 	\lVert \langle hD\rangle^{-1} u\rVert_{H^{1}_{\mathrm{scl}}(\Rb^{1+n})}^{2}\geq \frac{1}{2}\lVert\chi \langle hD\rangle^{-1}u \rVert_{H^{1}_{\mathrm{scl}}(\Rb^{1+n})}^{2}-\lVert\lb 1- \chi\rb \langle hD\rangle^{-1} u \rVert_{H^{1}_{\mathrm{scl}}(\Rb^{1+n})}^{2}.
 	\]
 	
 	Since $(1-\chi)\langle hD\rangle^{-1}$ is a smoothing semiclassical pseudodifferential operator, taking $h$ small enough, and arguing as in the proof of the Carleman estimate, we  get,

 	\begin{align*}
 	\lVert \Lc_{\vp} u\rVert_{H^{-1}_{\mathrm{scl}}(\Rb^{1+n})}^{2}\geq Ch^{2}\lVert u\rVert_{L^{2}(\Rb^{1+n})}^{2}.
 	\end{align*}
 	Cancelling out the $h^{2}$ term, we finally have,
 	\[
 	\lVert u\rVert_{L^{2}(\Rb^{1+n})}\leq \frac{C}{h}\lVert  \Lc_{\vp}u\rVert_{H^{-1}_{\mathrm{scl}}(\Rb^{1+n})}.
 	\]
 	This completes the proof.
 \end{proof}
 
 \begin{proposition}\label{Existence of solution for conjugated equation in adjoint problem}
 	Let $\vp$, $\Ac$ and $q$ be as in Theorem \ref{Boundary Carleman estimate}.  For $h>0$ small enough and $v\in L^{2}(Q)$, there exists a solution $u\in H^{1}(Q)$ of 
 	\[ 
 	\Lc_{\vp}u=v,
 	\]
 	satisfying the estimate
 	\[
 	\lVert u\rVert_{H^{1}_{\mathrm{scl}}(Q)} \leq \frac{C}{h}\lVert v\rVert_{L^{2}(\Rb^{1+n})},
 	\] 
 	where $C>0$ is a constant independent of $h$.
 	
 \end{proposition}
 \begin{proof} The proof uses standard functional analysis arguments.  
 	Consider the space $S:=\left\{ \Lc_{\vp}^{*}u: u\in C_{c}^{\infty}(Q)\right\}$ as a subspace of $H^{-1}(\mathbb{R}^{1+n})$ and define a linear form $L$ on $S$ by 
 	\begin{eqnarray*}
 		L( \Lc_{\vp}^{*}z)=\int\limits_{Q}z(t,x)v(t,x)\D x \D t, \mbox{ for } z\in C_{c}^{\infty}(Q).
 	\end{eqnarray*}
 	This is a well-defined continuous linear functional by the Carleman estimate \eqref{Carleman estimate in Sobolev space of negative order adjoint operator}{.}\color{black}\ 
 	We have 
 	\begin{eqnarray*}
 		\left|L( \Lc_{\vp}^{*}z)\right|\leq \lVert z\rVert_{L^{2}(Q)}\lVert v\rVert_{L^{2}(Q)}\leq \frac{C}{h}\lVert v\rVert_{L^{2}(Q)}\lVert \Lc_{\vp}^{*}z\rVert_{H^{-1}_{\mathrm{scl}}(\Rb^{1+n})}, z\in C_{c}^{\infty}(Q).
 	\end{eqnarray*}
 	By Hahn-Banach theorem, we {can}  extend $L$ to $H^{-1}(\Rb^{1+n})$ (still denoted as $L$) and it satisfies  $\lVert L\rVert\leq \frac{C}{h} \left|\left|v\right|\right|_{L^{2}(Q)}$. By Riesz representation theorem, there exists a unique  $u\in H^{1}(\mathbb{R}^{1+n})$  such that
 	\begin{eqnarray*}
 		L(f)=\langle f,u\rangle_{L^{2}(\Rb^{1+n})} \mbox{ for all } f\in H^{-1}(\Rb^{1+n}) \mbox{ with } \lVert u\rVert_{H^{1}_{\mathrm{scl}}(\Rb^{1+n})}\leq \frac{C}{h}\lVert v\rVert_{L^{2}(Q)}.
 	\end{eqnarray*}
 	Taking $f= \Lc_{\vp}^{*}z $, for $z\in C_{c}^{\infty}(Q)$, we get 
 	\begin{align*}
 	L(\Lc_{\vp}^{*}z)=\langle \Lc_{\vp}^{*}z,u\rangle_{L^{2}(\Rb^{1+n})}=\langle z,\Lc_{\vp}u\rangle_{L^{2}(\Rb^{1+n})}.
 	\end{align*}
 	Therefore for all $z\in C_{c}^{\infty}(Q)$,
 	\[
 	\langle  z,\Lc_{\vp} u\rangle =\langle z,v\rangle.
 	\]
 	Hence 
 	\begin{eqnarray*}
 		\Lc_{\vp}u=v \mbox{ in } Q \mbox{ with }  \lVert u\rVert_{H^{1}_{\mathrm{scl}}(Q)}\leq \frac{C}{h}\lVert v\rVert_{L^{2}(Q)}.
 	\end{eqnarray*}
 	This completes the proof of the proposition.
 \end{proof}
 \section{Construction of geometric optics solutions}\label{GO solutions}
 In this section we construct geometric optics solutions for $\Lc_{\mathcal{A},q}u=0$ and its adjoint operator $\Lc_{\mathcal{A},q}^{*}u= \Lc_{-\Ac,\overline{q}}u=0$.
 
 \begin{proposition} 
 	Let $\Lc_{\Ac,q}$ be as in \eqref{definition of operator}. 
 	\begin{enumerate}
 		\item (Exponentially decaying solutions) There exists an $h_{0}>0$ such that for all $0<h\leq h_{0}$, we can find $v\in H^{1}(Q)$ satisfying $\Lc_{-\Ac,\overline{q}}v=0$ of the form 
 		\begin{equation}\label{Exponentially decaying solution}
 		v_{d}(t,x)= e^{-\frac{\vp}{h}}\lb B_{d}(t,x) + hR_{d}(t,x;h)\rb,
 		\end{equation}
 		where $\vp(t,x)=t+x\cdot \omega$, 
 		\begin{equation}\label{Expression for Bd}
 		B_{d}(t,x)=\exp\left(-\int\limits_{0}^{\infty}(1,-\omega)\cdot\mathcal{A}(t+s,x-s\omega)\D s\right)
 		\end{equation}
 		with $\zeta\in(1,-\omega)^{\perp}$ and $R_{d}\in H^{1}(Q)$ satisfies
 		\begin{equation}\label{estimate for Remainder term1}
 		\lVert R_{d}\rVert_{H^{1}_{\mathrm{scl}}(Q)}\leq C .
 		\end{equation}
 		\item (Exponentially growing solutions) There exists an $h_{0}>0$ such that for all $0<h\leq h_{0}$, we can find $v\in H^{1}(Q)$ satisfying $\Lc_{\Ac,q}v=0$ of the form 
 		\begin{equation}\label{Exponentially growing solution}
 		v_{g}(t,x)= e^{\frac{\vp}{h}}\lb B_{g}(t,x) + hR_{g}(t,x;h)\rb,
 		\end{equation}
 		where $\vp(t,x)=t+x\cdot \omega$, 
 		\begin{equation}\label{Expression for Bg}
 		B_{g}(t,x)=e^{-i\zeta\cdot(t,x)}\exp\left(\int\limits_{0}^{\infty}(1,-\omega)\cdot\mathcal{A}(t+s,x-s\omega)\D s\right)
 		\end{equation}
 		with $\zeta\in(1,-\omega)^{\perp}$ and $R_{g}\in H^{1}(Q)$ satisfies
 		\begin{equation}\label{estimate for Remainder term2}
 		\lVert R_{g}\rVert_{H^{1}_{\mathrm{scl}}(Q)}\leq C .
 		\end{equation}
 	\end{enumerate}
 \end{proposition}
 \begin{proof}
 	We have 
 	\[
 	\Lc_{A,q}v = \Box v+2A_{0} \PD_{t}v-2A\cdot \n_{x}v+ \lb \PD_{t}A_{0}-\n_{x}\cdot A +|A_{0}|^{2}-|A|^{2}+q\rb v.
 	\]
 	Letting $v$ of the form 
 	\[
 	v(t,x) = e^{\frac{\vp}{h}}\lb B_{g}+ hR_{g}\rb,
 	\]
 	and setting the term involving $h^{-1}$ to be $0$, we get,
 	\[
 	(1,-\omega)\cdot \lb \n_{t,x} B_{g} +\lb A_{0},A\rb B_{g}\rb=0.
 	\]
 	One solution of this equation is 
 	\[
 	B_{g}(t,x)=\exp\lb \int\limits_{0}^{\infty} \lb 1,-\omega\rb\cdot \Ac(t+s,x-s\omega) \D s\rb.
 	\]
 	Alternately, another solution is 
 	\[
 	B_{g}(t,x)= e^{-\I \zeta \cdot (t,x)}\exp\lb \int\limits_{0}^{\infty} \lb 1,-\omega\rb\cdot \Ac(t+s,x-s\omega) \D s\rb,
 	\]
 	provided $\zeta \in (1,-\omega)^{\perp}$.
 	
 	Now we have 
 	\[
 	\Lc^{*}_{\Ac,q}=\Lc_{-\Ac,\overline{q}}.
 	\]
 	For this adjoint operator, the equation satisfied by $B_{d}$ is 
 	\[
 	(1,-\omega)\cdot \lb \n_{t,x} B_{g} -\lb A_{0},A\rb B_{g}\rb=0.
 	\]
 	We let 
 	\[
 	B_{d}(t,x)=\exp\lb-\int\limits_{0}^{\infty} \lb 1,-\o\rb\cdot \Ac(t+s,x-s\o) \D s\rb.
 	\]
 	
 	Now $R_{g}$ satisfies 
 	\[
 	\Lc_{\vp} R_{g}=-h\Lc_{\Ac,q}B_{g}.
 	\]
 	Then using the estimate in Proposition \ref{Existence of solution for conjugated equation in adjoint problem}, we get that 
 	\[
 	\lVert R_{g}\rVert_{H^{1}_{\mathrm{scl}}(Q)}\leq C \lVert \Lc_{\Ac,q} B_{g}\rVert_{L^{2}(Q)}.
 	\]
 	Similarly, 
 	\[
 	\lVert R_{d}\rVert_{H^{1}_{\mathrm{scl}}(Q)}\leq C \lVert \Lc_{-\Ac,\overline{q}} B_{d}\rVert_{L^{2}(Q)}.
 	\]
 	The proof is complete.
 \end{proof}

 \section{Integral Identity}\label{Integral identity}
 In this section, we derive an integral identity involving the coefficients $\Ac$ and $q$ using the geometric optics solutions described in the previous section. 
 
 Let $u_{i}$ be the solutions to the following initial boundary value problems with vector field coefficient $\Ac^{(i)}$ and scalar potential $q_{i}$ for $i=1,2$.
 \begin{align}\label{Equation for ui}
 \begin{aligned}
 \begin{cases}
 &\Lc_{\mathcal{A}^{(i)},q_{i}}u_{i}(t,x)=0;\ (t,x)\in Q\\
 &u_{i}(0,x)=\phi(x),\ \partial_{t}u_{i}(0,x)=\psi(x); \ x\in\Omega\\
 &u_{i}(t,x)=f(t,x),\ (t,x)\in \Sigma.
 \end{cases}
 \end{aligned}
 \end{align} 
 Let us denote 
 \begin{align}\label{Difference defn}
 \notag &u(t,x):=(u_{1}-u_{2})(t,x)\\
 \notag &\mathcal{A}(t,x):=\lb \mathcal{A}^{(2)}-\mathcal{A}^{(1)}\rb(t,x):=(A_{0}(t,x),A_{1}(t,x),\cdots,A_{n}(t,x))\\
 &\wt{q}_{i}:= \partial_{t}A_{0}^{(i)}-\nabla_{x}\cdot\mathcal{A}^{(i)}+|A_{0}^{(i)}|^{2}-|\mathcal{A}^{(i)}|^{2}+q_{i}\\
 \notag &\wt{q}:= \wt{q}_{2}-\wt{q}_{1}.
 \end{align}
 Then $u$ is the solution to the following initial boundary value problem: 
 \begin{align}\label{Equation for u linear}
 \begin{aligned}
 \begin{cases}
 &\Lc_{\mathcal{A}^{(1)},q_{1}}u(t,x)
 =-2A\cdot\nabla_{x}u_{2}+2A_{0}\partial_{t}u_{2}+\wt{q}u_{2} \\
 & u(0,x)=\partial_{t}u(0,x)=0,\ x\in\Omega\\
 &u|_{\Sigma}=0.
 \end{cases}
 \end{aligned}
 \end{align}
 Let $v(t,x)$ of the form given by \eqref{Exponentially decaying solution} be the solution to following equation 
 \begin{align}\label{adjoint equation for u1}
 \Lc_{-\mathcal{A}^{(1)},\overline{q}_{1}}v(t,x)
 =0\  \text{in}\ Q.
 \end{align}
 Also let $u_{2}$ of the form given by  \eqref{Exponentially growing solution} be solution to the following equation
 \begin{align}\label{equation for u2}
 \begin{aligned}\
 \Lc_{\mathcal{A}^{(2)},q_{2}}u_{2}(t,x)
 &=0,\   \text{in}\   Q.
 \end{aligned}
 \end{align}
 By the well-posedness result from \cite{Katchalov_Kurylev_Lassas_Book_2001, Lasiecka_Lions_Triggiani_Nonhomogeneous_BVP_hyperbolic_1986}, we have $u\in H^{1}(Q)$ and $\partial_{\nu}u\in L^{2}(\Sigma)$.

 Now we multiply  \eqref{Equation for u linear} by $\overline{v(t,x)}\in H^{1}(Q)$ and integrate over $Q$. We get, after integrating by parts, taking into account the following: $u|_{\Sigma}=0$, $u(T,x)=0$, $\partial_{\nu}u|_{G}=0$, $u|_{t=0}=\PD_{t}u|_{t=0}=0$ and $\Ac^{(1)}$ is compactly supported in $Q$:   
 \begin{align*}
 \int\limits_{Q}\Lc_{\Ac^{(1)},q_{1}}u(t,x) \overline{v(t,x)} \D x \D t-\int\limits_{Q}u(t,x) \overline{\Lc_{-\Ac^{(1)},\overline{q}_{1}}v(t,x)} \D x \D t&=\int\limits_{\Omega}\partial_{t}u(T,x) \overline{v(T,x)}\D x\\
 &-\int\limits_{\Sigma\setminus{G}}\partial_{\nu}u(t,x)\overline{v(t,x)}\D S_{x}\D t.
 \end{align*}
 Now using the fact that $\Lc_{-\mathcal{A}^{(1)},\overline{q}_{1}}v(t,x)=0$ in $Q$ and 
 \[
 \Lc_{\Ac^{(1)},q_{1}}u(t,x)=-2A\cdot\nabla_{x}u_{2}+2A_{0}\partial_{t}u_{2}+\wt{q}u_{2},\]
 we get,
 \begin{align*}
 \int\limits_{Q}\left(-2A\cdot\nabla_{x}u_{2}+2A_{0}\partial_{t}u_{2}+\wt{q}u_{2}\right)\overline{v(t,x)}\D x \D t&=\int\limits_{\Omega}\partial_{t}u(T,x) \overline{v(T,x)}\D x\\
 &-\int\limits_{\Sigma\setminus{G}}\partial_{\nu}u(t,x)\overline{v(t,x)}\D S_{x}\D t.
 \end{align*}
 \begin{lemma}
 	Let  $u_{i}$ for $i=1,2$ solutions to \eqref{Equation for ui} with $u_{2}$ of the form \eqref{Exponentially growing solution}. Let $u=u_{1}-u_{2}$, and $v$ be of the form \eqref{Exponentially decaying solution}. Then 
 	\begin{equation}\label{first remainder}
 	h\int\limits_{\O} \PD_{t}u(T,x) \overline{v(T,x)} \D x \to 0 \mbox{ as } h\to 0^{+}.
 	\end{equation}
 	\begin{equation}\label{second remainder}
 	h\int\limits_{\Sigma\setminus{G}}\partial_{\nu}u(t,x)\overline{v(t,x)}\D S_{x}\D t \to 0 \mbox{ as } h\to 0^{+}.
 	\end{equation}
 \end{lemma}
 \begin{proof}
 	
 	Using  \eqref{Exponentially decaying solution}, \eqref{estimate for Remainder term1} and Cauchy-Schwartz inequality, we get 
 	\begin{align*}
 	&\left|h\int\limits_{\Omega}\partial_{t}u(T,x)\overline{v(T,x)}\D x\right|\leq \int\limits_{\Omega}h\left|\partial_{t}u(T,x)e^{-\frac{\vp(T,x)}{h}}\overline{\left(B_{d}(T,x)+h R_{d}(T,x)\right)}\right|\D x \\
 	&\leq C\left(\int\limits_{\Omega}h^{2}\left|\partial_{t}u(T,x)e^{-\frac{\vp(T,x)}{h}}\right|^{2}\D x\right)^{\frac{1}{2}}\left(\int\limits_{\Omega}\left|e^{-i\xi\cdot(T,x)}+h\overline{R_{d}(T,x)}\right|^{2}\D x\right)^{\frac{1}{2}}\\
 	&\leq C\left(\int\limits_{\Omega}h^{2}\left|\partial_{t}u(T,x)e^{-\frac{\vp(T,x)}{h}}\right|^{2}\D x\right)^{\frac{1}{2}}\left(1+\lVert hR_{d}(T,\cdot)\rVert^{2}_{L^{2}(\Omega)}\right)^{\frac{1}{2}}\\
 	&\leq C \left(\int\limits_{\Omega}h^{2}\left|\partial_{t}u(T,x)e^{-\frac{\vp(T,x)}{h}}\right|^{2}\D x\right)^{\frac{1}{2}}.
 	\end{align*}
 	
 	Now using the boundary Carleman estimate \eqref{Boundary Carleman estimate}, 	we get, 
 	
 	\begin{align*}
 	h\int\limits_{\Omega}\left|\partial_{t}u(T,x)e^{-\frac{\vp(T,x)}{h}}\right|^{2}\D x& \leq C\lVert he^{-\vp/h} \Lc_{\Ac^{(1)},q_{1}} u\rVert_{L^{2}(Q)}^{2}\\
 	&=C\lVert he^{-\vp/h}\lb 2 A_{0}\PD_{t} u_{2}-2A\cdot \n_{x} u_{2}+\wt{q} u_{2}\rb\lVert^{2}_{L^{2}(Q)}.
 	\end{align*}
 	Now substituting \eqref{Exponentially growing solution} for $u_{2}$, we get,
 	\begin{align*}
 	he^{-\vp/h}\lb 2 A_{0}\PD_{t} u_{2}-2A\cdot \n_{x} u_{2}+\wt{q} u_{2}\rb& = \lb 2A_{0}\PD_{t}\vp-2A\cdot \n_{x}\vp +\wt{q}\rb \lb B_{g}+hR_{g}\rb\\
 	&+ h\lb 2A_{0}\PD_{t}-2A\cdot \n_{x}+\wt{q}\rb \lb B_{g}+hR_{g}\rb
 	\end{align*}	
 	Therefore 
 	\[
 	\lVert he^{-\vp/h}\lb 2 A_{0}\PD_{t} u_{2}-2A\cdot \n_{x} u_{2}+\wt{q} u_{2}\rb\lVert^{2}_{L^{2}(Q)}\leq C,
 	\]
 	uniformly in $h$.\\
 	
 	Thus, we have 
 	\[
 	\lb h^{2}\int\limits_{\Omega}\left|\partial_{t}u(T,x)e^{-\frac{\vp(T,x)}{h}}\right|^{2}\D x\rb^{\frac{1}{2}}\leq C\sqrt{h}.
 	\]
 	
 	\color{black}
 	Therefore
 	\[
 	h\int\limits_{\O} \PD_{t}u(T,x) \overline{v(T,x)} \D x \to 0 \mbox{ as } h\to 0^{+}.
 	\]
 	
 	For $\ve>0$, define 
 	\[
 	\PD\Omega_{+,\ve,\omega}=\{x\in \PD \Omega: \nu(x) \cdot \omega > \ve\}.
 	\]
 	and
 	\[
 	\Sigma_{+,\ve,\omega}=(0,T)\times \PD \Omega_{+,\ve,\omega}.
 	\]
 	
 	Next we prove \eqref{second remainder}.  Since $\Sigma\setminus G\subseteq \Sigma_{+,\ve,\omega}$ for all $\omega$ such that $|\omega-\omega_{0}|\leq \ve$, substituting $v=v_{d}$ from \eqref{Exponentially decaying solution}, in \eqref{second remainder} we have 
 	\begin{align*}
 	\begin{aligned}
 	&\left|\int\limits_{\Sigma\setminus{G}}\partial_{\nu}u(t,x)v(t,x)\D S_{x}\D t\right|\leq\int\limits_{\Sigma{+,\ve,\omega}}\left|\partial_{\nu}u(t,x)e^{-\frac{\vp}{h}}\left(B_{d}+hR_{d}\right)(t,x)\right|\D S_{x}\D t\\
 	&\leq C\left(1+\lVert hR_{d}\rVert_{L^{2}(\Sigma)}^{2}\right)^{\frac{1}{2}}\left(\int\limits_{\Sigma{+,\ve,\omega}}\left|\partial_{\nu}u(t,x)e^{-\frac{\vp}{h}}\right|^{2}\D S_{x}\D t\right)
 	\end{aligned}
 	\end{align*}
 	with $C>0$ is independent of $h$ and this inequality holds for all $\omega $ such that $|\omega -\omega_{0}|\leq \ve$. Next by trace theorem, we have that  $\lVert R_{d}\rVert_{L^{2}(\Sigma)}\leq C\lVert R_{d}\rVert_{H^{1}_{\mathrm{scl}(Q)}}$.
 	
 	Using this, we get 
 	\begin{align*}
 	\left|\int\limits_{\Sigma\setminus{G}}\partial_{\nu}u(t,x)v(t,x)\D S_{x}\D t\right|\leq C\left(\int\limits_{\Sigma{+,\ve,\omega}}\left|\partial_{\nu}u(t,x)e^{-\frac{\vp}{h}}\right|^{2}\D S_{x}\D t\right)^{\frac{1}{2}}.
 	\end{align*}
 	
 	Now 
 	\begin{align*}
 	\int\limits_{\Sigma_{+},\ve,\omega}\left|\partial_{\nu}u(t,x)e^{-\frac{\vp}{h}}\right|^{2}\D S_{x}\D t& =\frac{1}{\ve}\int \limits_{\Sigma_{+},\ve,\omega}\ve \left|\partial_{\nu}u(t,x)e^{-\frac{\vp}{h}}\right|^{2}\D S_{x}\D t\\
 	&\leq \frac{1}{\ve}\int\limits_{\Sigma_{+},\ve,\omega} \PD_{\nu}\vp \left|\partial_{\nu}u(t,x)e^{-\frac{\vp}{h}}\right|^{2}\D S_{x}\D t.
 	\end{align*}
 	Using the boundary Carleman estimate \eqref{Boundary Carleman estimate}, we have 
 	\[
 	\frac{h}{\ve}\int\limits_{\Sigma_{+},\ve,\omega} \PD_{\nu}\vp \left|\partial_{\nu}u(t,x)e^{-\frac{\vp}{h}}\right|^{2}\D S_{x}\D t \leq C\lVert he^{-\vp/h} \Lc_{\Ac^{(1)},q_{1}} u\rVert_{L^{2}(Q)}^{2}.
 	\]
 	We now proceed as before to conclude that 
 	\[
 	h\int\limits_{\Sigma\setminus{G}}\partial_{\nu}u(t,x)\overline{v(t,x)}\D S_{x}\D t \to 0 \mbox{ as } h\to 0^{+}.
 	\]
 \end{proof}

 \section{Proof of Theorem \ref{Main Theorem}}\label{Proof for Uniqueness}
 In this section, we prove the uniqueness results.
 \subsection{Recovery of vector potential $\Ac$}\label{Recovery of vector potential}
 We consider the integral, 
 \[
 \int\limits_{Q}\left(-2A\cdot\nabla_{x}u_{2}+2A_{0}\partial_{x}u_{2}+\wt{q}u_{2}\right )(t,x)\overline{v(t,x)}\D x \D t.
 \]
 Substituting \eqref{Exponentially growing solution} for $u_{2}$ and \eqref{Exponentially decaying solution} for $v$ into the above equation, and letting $h\to 0^{+}$, we arrive at 
 \[
 \int\limits_{Q}\left(-\omega\cdot A+A_{0}\right)\overline{B_{d}(t,x)}B_{g}(t,x)\D x \D t=0 \mbox{ for all } \o \in \Sb^{n-1} \mbox{ such that } |\omega-\omega_{0}|\leq \ve.
 \] 
 Denote $\wt{\omega}:=(1,-\omega)$, $\mathcal{A}=(A_{0},A)$,
 and using the expressions for $B_{d}$ and $B_{g}$,  see \eqref{Exponentially decaying solution} and \eqref{Exponentially growing solution}, 
 we get 
 \[
 J:=\int\limits_{\mathbb{R}^{1+n}}\wt{\omega}\cdot\mathcal{A}(t,x)e^{-i\xi\cdot(t,x)}\exp\left(\int\limits_{0}^{\infty}\wt{\omega}\cdot\mathcal{A}(t+s,x-s\omega)\right)\D x \D t=0,
 \]
 where $\xi\cdot(1,-\omega)=0$ for all  $\omega\ \text{with} \
 |\omega-\omega_{0}|<\ve$. We decompose  $\mathbb{R}^{1+n}=\mathbb{R}(1,-\omega)\oplus(1,-\omega)^{\perp}$. We then get 
 \begin{align*}
 J&=\int\limits_{(1,-\omega)^{\perp}}e^{-\I\xi\cdot k}\left(\int\limits_{\mathbb{R}}\wt{\omega}\cdot\mathcal{A}(k+\tau(1,-\omega))\exp\left(\int\limits_{\tau}^{\infty}\wt{\omega}\cdot\mathcal{A}(k+s(1,-\omega))\D s\right)\sqrt{2} \D \tau\right)\D k.
 \intertext{Here $\D k$ is the Lebesgue measure on the hyperplane $(1,-\omega)^{\perp}$.}
 J&=-\sqrt{2}\int\limits_{(1,-\omega)^{\perp}}e^{-\I\xi\cdot k}\left(\int\limits_{\mathbb{R}}\partial_{\tau}\exp\left(\int\limits_{\tau}^{\infty}\wt{\omega}\cdot\mathcal{A}(k+s(1,-\omega))\D s\right)\D\tau\right)\D k\\
 &=-\sqrt{2}\int\limits_{(1,-\omega)^{\perp}}e^{-i\xi\cdot k}\left(1-\exp\left(\int\limits_{\mathbb{R}}\wt{\omega}\cdot\mathcal{A}(k+s(1,-\omega))\D s\right)\right)\D k\\
 &=-\sqrt{2}\ \mathcal{F}_{(1,-\omega)^{\perp}}\left( 1-\exp\left(\int\limits_{\mathbb{R}}\wt{\omega}\cdot\mathcal{A}(k+s(1,-\omega))\D s\right)\right)(\xi).
 \end{align*}
 \color{black}
 Since the integral $J=0$, we have that 
 \[
 \mathcal{F}_{(1,-\omega)^{\perp}}\left( 1-\exp\left(\int\limits_{\mathbb{R}}\wt{\omega}\cdot\mathcal{A}(k+s(1,-\omega))\D s\right)\right)(\xi)=0,\  k\in (1,-\omega)^{\perp}.
 \]
 
 This gives us 
 \begin{align*}
 \exp\left(\int\limits_{\mathbb{R}}\wt{\omega}\cdot\mathcal{A}(k+s(1,-\omega))\D s\right)=1, \mbox{ for all } \ k\in (1,-\omega)^{\perp} \mbox{ and all } \omega\ \text{with} \
 |\omega-\omega_{0}|<\ve.
 \end{align*}
 Thus we deduce that
 \begin{align}\label{final equation for vector field}
 \int\limits_{\mathbb{R}}\wt{\omega}\cdot\mathcal{A}(t+s,x-s\omega)\D s=0,\ \ \ (t,x)\in (1,-\omega)^{\perp}\mbox{ and } \mbox{for all } \omega\ \text{with} \
 |\omega-\omega_{0}|<\ve.
 \end{align}
 Now we show that the orthogonality condition  $(t,x)\in (1,-\omega)^{\perp}$, can be removed using  a change of variables as used in \cite{Salazar_time-dependent_first_order_perturbation_2013}. 
 
 Consider any $(t,x)\in \Rb^{1+n}$. Then 
 \[
 \lb \frac{t+x\cdot \omega}{2}, x+\frac{(t-x\cdot \omega)\omega}{2}\rb
 \] 
 is a point on $(1,-\omega)^{\perp}$, and we have
 \[
 \int\limits_{\Rb} \wt{\omega}\cdot \Ac\lb \frac{t+x\cdot \omega}{2} +\wt{s}, x+\frac{(t-x\cdot \omega)\omega}{2} -\wt{s}\omega\rb \D \wt{s} =0.
 \] 
 Consider the following change of variable in the above integral:  
 \[s= \frac{x\cdot \omega-t}{2}+\wt{s}.
 \]
 Then we have 
 \[
 \int\limits_{\Rb} \wt{\omega}\cdot \Ac\lb t+s, x-s\omega\rb \D s=0.
 \]
 Therefore, we have that 
 \[
 \int\limits_{\mathbb{R}}(1,-\omega)\cdot\mathcal{A}(t+s,x-s\omega)\D s=0 \mbox{ for all } (t,x)\in \mathbb{R}^{1+n} \mbox{ and for all } \omega \text{ with } 
 |\omega-\omega_{0}|<\ve.
 \]

 To conclude the uniqueness result for $\Ac$, we prove the following lemma.
 
 \begin{lemma}\label{support theorem}
 	Let $n\geq 3$ and $F=(F_{0},F_{1},\cdots,F_{n})$ be a real-valued vector field whose components are  $C_{c}^{\infty}(R^{1+n})$ functions.  Suppose 
 	\begin{align*}
 	LF(t,x,\omega):=\int\limits_{\mathbb{R}}(1,\omega)\cdot F(t+s, x+s\omega)\D s=0
 	\end{align*}
 	for all  $\omega\in\mathbb{S}^{n-1}$ near  a fixed $\omega_{0}\in \mathbb{S}^{n-1}$ and for all $(t,x)\in \Rb^{1+n}$.
 	Then {there exists}\color{black}\  a $\Phi\in C_{c}^{\infty}(\mathbb{R}^{1+n})$ such that $F(t,x)=\nabla_{t,x}\Phi(t,x)$. 
 \end{lemma}
 \begin{proof}
 	The proof follows {from arguments similar} \color{black} to the ones used in \cite{Stefanov_Support_Theorem_Lorentzian_Manifold_2017,Siamak_Support_theorem_2017}, where support theorems involving light ray transforms have been proved.
 	
 	Denote $\omega=(\omega^{1},\cdots,\omega^{n})\in \Sb^{n-1}$. We write
 	\[
 	LF(t,x,\omega)=\int\limits_{\mathbb{R}}\sum_{i=0}^{n}\omega^{i}F_{i}(t+s,x+s\omega)\D s,\mbox{ where } \omega^{0}=1.
 	\]
 	Let $\eta=(\eta_{0},\eta_{1},\cdots,\eta_{n})\in \Rb^{n+1}$ be arbitrary. We have 
 	\begin{equation}\label{I equation}
 	(\eta\cdot\nabla_{t,x})LF(t,x,\omega)=\int\limits_{\mathbb{R}}\sum_{i,j=0}^{n}\omega^{i}\eta_{j}\partial_{j}F_{i}(t+s,x+s\omega) \D s.
 	\end{equation}
 	By fundamental theorem of calculus, we have 
 	\begin{align*}
 	\int\limits_{\mathbb{R}}\frac{\D }{\D s}(\eta\cdot F)(t+s,x+s\omega)\D s=0.
 	\end{align*}
 	But 
 	\[
 	\frac{\D }{\D s}(\eta\cdot F)(t+s,x+s\omega)=
 	\sum_{i,j=0}^{n}\omega^{i}\eta_{j}\partial_{i}F_{j}(t+s,x+s\omega).
 	\]
 	with $\partial_{0}=\partial_{t}$ and $\partial_{j}=\partial_{x_{j}}$ for  $j=1,2,\cdots,n$.
 	Therefore 
 	\begin{equation}\label{II equation}
 	\int\limits_{\mathbb{R}}\sum_{i,j=0}^{n}\omega^{i}\eta_{j}\partial_{i}F_{j}(t+s,x+s\omega)\D s=0.
 	\end{equation}
 	Subtracting \eqref{II equation} from \eqref{I equation}, we get,
 	\[
 	\lb \eta\cdot \n_{t,x}\rb LF(t,x,\omega)=\int\limits_{\mathbb{R}}\sum_{i,j=0}^{n}\omega^{i}\eta_{j}\left(\partial_{j}F_{i}-\partial_{i}F_{j}\right)(t+s,x+s\omega)\D s.
 	\]
 	Since $LF(t,x,\omega)=0$ for all $\omega$ near $\omega_{0}$, and for all $(t,x)\in \Rb^{1+n}$, we have,
 	\begin{equation}\label{transverse ray transform}
 	Ih(t,x,\omega,\eta):=\int\limits_{\mathbb{R}}\sum_{i,j=0}^{n}\omega^{i}\eta_{j}h_{ij}(t+s,x+s\omega)\D s=0
 	\end{equation}
 	 where $h$ is $\lb 1+n\rb\times\lb 1+n\rb$ matrix with entries $h_{ij}$ given by 
 	\[h_{ij}(t,x)=\lb \PD_{j}F_{i}-\PD_{i}F_{j}\rb(t,x) \ \text{for} \ 0\leq i,j\leq n.\]
 	\color{black}Next we will show that the $(n+1)$ dimensional Fourier transform 
 	$\widehat{h}_{ij}(\zeta)=0$ for all space-like vectors $\zeta$ near the set $\left\{\zeta:\ \zeta\cdot(1,\omega)=0, \omega \mbox{ near } \omega_{0}\right\}$. 
 	
 	We have 
 	\[
 	\omega^{i}\eta_{j}\wh{h}_{ij}(\zeta)=\int\limits_{\Rb^{1+n}} e^{-\I (t,x)\cdot \zeta}\omega^{i}\eta_{j} h_{ij}(t,x) \D t \D x,
 	\]
 	where $\omega,\eta$ are fixed and $\zeta\in (1,\omega)^{\perp}$.  Decomposing 
 	\[
 	\Rb^{1+n}= \Rb(1,\omega)+k, \mbox{ where } k\in (1,\omega)^{\perp},
 	\]
 	we get, 
 	\[
 	\omega^{i}\eta_{j}\wh{h}_{ij}(\zeta)=\sqrt{2}\int\limits_{(1,\omega)^{\perp} }e^{-\I k\cdot \zeta}\int\limits_{\Rb} \omega^{i}\eta_{j} h_{ij}(k+s(1,\omega)) \,\D s \D k.
 	\]
 	Using \eqref{transverse ray transform}, we get that, 
 	\begin{equation}\label{Final 1-form equation}
 	\sum\limits_{i,j=0}^{n}\omega^{i}\eta_{j} \wh{h}_{ij}(\zeta)=0, \mbox{ for all } \zeta\in (1,\omega)^{\perp}, \mbox{ for all } \eta\in \Rb^{1+n}, \mbox{ and for all } \omega \mbox{ near } \omega_{0} \mbox{ with } \omega^{0}=1.
 	\end{equation}
 	Let us take for $\eta$ the standard basis vectors in $\Rb^{1+n}$. 
 	
 	Now $\{e_{j}, 1\leq j\leq n\}$ be the standard basis of $\Rb^{n}$. Let $\omega_{0}= e_{1}$, and assume that $\zeta^{0}=(0,e_{2})$. Then this is a space-like vector that satisfies the condition $\zeta^{0}\in (1,\omega_{0})^{\perp}$. We will show that $\wh{h}_{ij}(\zeta^{0})=0$ for all $0\leq i,j\leq n$. Consider the collection of the following unit vectors for $3\leq k\leq n$: 
 	\[
 	\omega_{k}(\A)=\cos (\A) e_{1}+\sin (\A) e_{k}.
 	\]
 	Note that for each $3\leq k\leq n, \omega_{k}(\A)$ is near $\omega_{0}$ for $\A$ near $0$. Also $\zeta^{0}\in (1,\omega_{k}(\A))^{\perp}$ for all such $\A$ and for all $3\leq k\leq n$. Let $\eta$ be the collection of standard basis vectors in $\Rb^{1+n}$.
 	
 	Substituting the above vectors $\omega_{k}(\A)$ and $\eta$, we get the following equations: 
 	
 	\[
 	\wh{h}_{0j}(\zeta^{0}) +\cos (\A) \wh{h}_{1j}(\zeta^{0}) + \sin (\A)\wh{h}_{kj}(\zeta^{0})=0 \mbox{ for all } 3\leq k\leq n,  0\leq j\leq n \mbox{ and } \A \mbox{ near } 0.
 	\]
 	
 	From this we have that 
 	\[
 	\wh{h}_{0j}(\zeta^{0})=\wh{h}_{1j}(\zeta^{0})=\wh{h}_{kj}(\zeta^{0})=0 \mbox{ for all } 3\leq k\leq n \mbox{ and } 0\leq j\leq n.
 	\]
 	
 	Using the fact that 
 	\[
 	\widehat{h}_{ij}(\zeta^{0})=-\widehat{h}_{ji}(\zeta^{0}) \mbox{ for all } 0\leq i,j\leq n,
 	\] we get 
 	\begin{align}\label{vanishing of hij at zeta not}
 	\widehat{h}_{ij}(\zeta^{0})=0,\mbox{ for all }  0\leq i,j\leq n \mbox{ with } \zeta^{0}=(0,e_{2}).
 	\end{align}	
 	
 	Now our goal is to show that the Fourier transform $\wh{h}_{ij}(\zeta)$ vanishes for all space-like vectors $\zeta$ in a small enough neighborhood of $(0,e_{2})$. To show this, assume $\zeta=(\tau,\xi)$ be a space-like vector close to $(0,e_{2})$. Without loss of generality, assume that $\zeta$ is of the form 
 	\[
 	\zeta=(\tau,\xi)  \mbox{ where } |\xi|=1 \mbox{ and } |\tau|<1.
 	\]
 	For instance, we can take $\zeta$ as 
 	\[
 	\zeta = (-\sin \vp,\xi),
 	\]
 	where $\vp$ is close to $0$ and $\xi=(\xi_{1},\cdots, \xi_{n})$ is written in spherical coordinates as 
 	
 	\begin{align*}
 	& \xi_{1}=\sin \vp_{1}\cos \vp_{2}\\
 	& \xi_{2}=\cos \vp_{1}\\
 	&\xi_{3}=\sin \vp_{1} \sin \vp_{2}\cos \vp_{3}\\
 	&\vdots\\
 	&\xi_{n-1}=\sin \vp_{1}\cdots \sin \vp_{n-2}\cos \vp_{n-1}\\
 	& \xi_{n-1}=\sin \vp_{1}\cdots \sin \vp_{n-2}\sin \vp_{n-1}.
 	\end{align*}
 	Note that if $\vp_{1},\cdots,\vp_{n-1}$ are close to $0$, then $\xi$ is close to $e_{2}$.	
 	
 	Let 
 	\[
 	\omega_{\vp}=\cos (\vp) e_{1}+\sin (\vp) e_{2}.
 	\]
 	
 	Also let 
 	\[
 	\omega_{\vp}^{k}(\A)= \cos (\A)\cos (\vp) e_{1} + \sin (\vp) e_{2}+ \sin (\A) \cos (\vp) e_{k} \mbox{ for } k\geq 3.
 	\]
 	Since $\vp$ is close to $0$, and letting $\A$ close enough to $0$, we have that $\omega_{\vp}$ and $\omega_{\vp}^{k}$ are close to $\omega_{0}$.
 	
 	Now, let $A$ be an orthogonal transformation such that $A e_{2}= \xi$, where $\xi$ is as above. With this $A$, consider the vectors 
 	\[
 	\omega_{\zeta}=A \omega_{\vp} \mbox{ and } \omega_{\zeta}^{k}=A\omega_{\vp}^{k} \mbox{ for } k\geq 3.
 	\]
 	
 	We have that 
 	\[
 	\zeta \in (1,\omega_{\zeta})^{\perp}.
 	\]
 	To see this, note that $\zeta=(-\sin \vp, \xi)$ and we have 
 	\[
 	-\sin \vp+ \langle A e_{2},A\omega_{\vp}\rangle = -\sin \vp+ \langle e_{2},\omega_{\vp}\rangle = 0.
 	\]
 	Also similarly, we have that, 
 	\[
 	\zeta\in (1,\omega_{\zeta}^{k})^{\perp} \mbox{ for all } k\geq 3.
 	\]
 	Using the standard basis vectors for $\eta_{j}$ and the vectors $(1,\omega_{\zeta}^{k})$ in \eqref{Final 1-form equation}, we get, for all $k\geq 3$, 
 	\[
 	\wh{h}_{0j}(\zeta) +\sin \vp \lb \sum\limits_{i=1}^{n}a_{i2}\wh{h}_{ij}(\zeta)\rb + \cos \A \cos \vp \lb\sum \limits_{i=1}^{n}a_{i1}\wh{h}_{ij}(\zeta)\rb +\sin \A \cos \vp \lb\sum \limits_{i=1}^{n}a_{ik}\wh{h}_{ij}(\zeta)\rb =0
 	\]
 	This then implies that 
 	\begin{align*}
 	&\wh{h}_{0j}(\zeta) +\sin \vp \lb  \sum\limits_{i=1}^{n}a_{i2}\wh{h}_{ij}(\zeta)\rb=0\\
 	&\cos \vp \lb\sum \limits_{i=1}^{n}a_{i1}\wh{h}_{ij}(\zeta)\rb=0\\
 	&\cos \vp \lb\sum \limits_{i=1}^{n}a_{ik}\wh{h}_{ij}(\zeta)\rb =0
 	\end{align*}
 	Letting $j=0$, since $\wh{h}_{00}(\zeta)=0$, we have, 
 	\begin{align*}
 	\sum\limits_{i=1}^{n} a_{ij} \wh{h}_{i0}(\zeta)=0 \mbox{ for all } 1\leq j\leq n.
 	\end{align*}
 	Since $A$ is an invertible matrix, we have that $\wh{h}_{i0}(\zeta)=0$ for all $1\leq i\leq n$.
 	Now proceeding similarly and using the fact that $h_{ij}$ is alternating, we have that $\wh{h}_{ij}(\zeta)=0$ for all $0\leq i,j\leq n$. The same argument as above also works for $r\zeta$, where $\zeta$ is as above and $r>0$. 
 	
 	Since the support of all $h_{ij}$ is a compact subset of $\mathbb{R}^{1+n}$, by Paley-Wiener theorem, we have $\widehat{h}_{ij}(\zeta)=0$ $\forall$ $i,j=0,1,2,...,n$. Hence, by Fourier inversion formula, we see that $h_{ij}(t,x)=0$ $\forall$ $(t,x)\in\mathbb{R}^{1+n}$. In other words, the exterior derivative $\D F=0$.  Using Poincar\'e lemma, we have that there exists a $\Phi(t,x) \in C_{c}^{\infty}(\Rb^{1+n})$ such that $F=\nabla_{t,x}\Phi$.
 \end{proof}
  Now going back to the uniqueness result for $\Ac$, we have that $\Ac$ is compactly supported in $Q$ and since $Q$ is simply connected, by Poincar\'e Lemma, there exists a $\Phi\in C_{c}^{\infty}(Q)$ such that 
 	$\Ac=\n_{t,x}\Phi$. 
 
 {
 	\begin{remark}
 		The proof requires dimensions $n\geq 3$. In the case of 2-dimensions, knowledge of the light ray transform along directions $(1,-\omega)$ for $\omega$ near $-\omega_{0}$ is required as well for invertibility modulo potential fields; see \cite{Siamak_Support_theorem_2017}. Since we do not have this information, our approach will not give the result in $2$-dimensions. 
 	\end{remark}
 }
 
 \subsection{Recovery of potential $q$}\label{Uniqueness for q}
 In Section \ref{Recovery of vector potential}, we showed  that there exist a $\Phi$ such that $(\mathcal{A}_{2}-\mathcal{A}_{1})(t,x)=\nabla_{t,x}\Phi(t,x)$. After replacing the pair $(\mathcal{A}^{(1)},q_{1})$ by $(\mathcal{A}^{(3)},q_{3})$ where $\mathcal{A}^{(3)}=\mathcal{A}^{(1)}+\nabla_{t,x}\Phi$ and $q_{3}=q_{1}$, we conclude that $\mathcal{A}^{(3)}=\mathcal{A}^{(2)}$. Therefore substituting \eqref{Exponentially growing solution} for $u_{2}$ and \eqref{Exponentially decaying solution} for $v$, and letting $h\to 0^{+}$ in 
 \[
 \int \limits_{Q} q(t,x) u_{2}(t,x) v(t,x) \D x \D t=\int\limits_{\Omega}\partial_{t}u(T,x) v(T,x)\D x\\
 -\int\limits_{\Sigma\setminus{G}}\partial_{\nu}u(t,x)\overline{v(t,x)}\D S_{x}\D t,
 \]  we get,
 \[
 \int\limits_{\Rb^{1+n}} q(t,x) e^{-\I \xi \cdot (t,x)} \D x \D t=0 \mbox{ for all } \xi \in (1,-\omega)^{\perp} \mbox{ and } \omega \mbox{ near } \omega_{0}.
 \]

 The set of all $\xi$ such that $\xi\in (1,-\omega)^{\perp}$ for $\omega$ near $\omega_{0}$ forms an open cone and since $q\in L^{\infty}(Q)$ has compact support, using Paley-Wiener theorem we conclude that $q_{1}(t,x)=q_{2}(t,x)$ for all $(t,x)\in Q$. This completes the proof of Theorem \ref{Main Theorem}.

 \section*{Acknowledgments}
 
 The authors thank Plamen Stefanov and Siamak RabieniaHaratbar for very useful discussions regarding the light ray transform.  MV thanks Gen Nakamura and Guang-Hui Hu for discussions on this problem. 
 VK was supported in part by NSF grant DMS 1616564. Both authors  benefited
 from the support of Airbus Corporate Foundation Chair grant titled ``Mathematics of Complex Systems''
 established at TIFR CAM and TIFR ICTS, Bangalore, India.

\end{document}